\documentclass[11pt]{amsart}
\usepackage{amsmath}
\usepackage{amssymb}
\usepackage{latexsym}
\usepackage{graphicx}

\newtheorem{theorem}{Theorem}
\newtheorem{corollary}{Corollary}
\newtheorem{lemma}{Lemma}
\theoremstyle{remark}

\begin{document}

\title[Coefficient determinants involving many Fekete-Szeg$\ddot{O}$-type parameters]{\large Coefficient determinants involving many Fekete-Szeg$\ddot{O}$-type parameters}

\author[K. O. Babalola]{K. O. BABALOLA}

\begin{abstract}
We extend our definition (in a recent paper \cite{KB}) of the coefficient determinants of analytic mappings of the unit disk to include many Fekete-Szeg$\ddot{o}$-type parameters, and compute the best possible bounds on certain specific determinants for the choice class of starlike functions. 
\end{abstract}

\maketitle


\medskip

\section{Hankel Determinants with Parameters}
In a recent paper \cite{KB}, we extended the definition of the well known Hankel determinant for coefficients of analytic mappings to include the also well known Fekete-Szeg$\ddot{o}$ parameter as follows:\vskip 2mm

{\sc Definition A.} Let $\lambda$ be a nonnegative real number. Then for integers $n\geq 1$ and $q\geq 1$, we define the $q$-th Hankel determinants with Fekete-Szeg$\ddot{o}$ parameter $\lambda$, that is $H_q^\lambda(n)$, as
$$H_q^\lambda(n)=
\begin{vmatrix}
a_n & a_{n+1} & \cdots & \lambda a_{n+q-1}\\
a_{n+1} & \cdots & \cdots & \vdots\\
\vdots & \vdots & \vdots & \vdots\\
a_{n+q-1} & \cdots & \cdots & a_{n+2(q-1)}
\end{vmatrix}$$
with the well known Hankel determinant being the case $\lambda=1$.\vskip 2mm

We also defined other similar determinants as:\vskip 2mm

{\sc Definition B.} Let $\lambda$ be a nonnegative real number. Then for integers $n\geq 1$ and $q\geq 1$, we define the $B_q^\lambda(n)$ determinants as
$$B_q^\lambda(n)=
\begin{vmatrix}
a_n & a_{n+1} & \cdots & a_{n+q-1}\\
a_{n+q} & a_{n+q+1} & \cdots & a_{n+2q-1}\\
a_{n+2q} & a_{n+2q+1} & \cdots & a_{n+3q-1}\\
\vdots & \vdots & \vdots & \vdots\\
a_{n+q(q-1)} & \cdots & \cdots & \lambda a_{n+q^2-1}
\end{vmatrix}.$$

Our motivation for those new definitions lie in the fact that, for function classes defined by other function classes (for example the classes of close-to-star, close-to-convex, quasi-convex, $\alpha$-starlike, $\alpha$-convex, $\alpha$-close-to-star, $\alpha$-close-to-convex whose definitions involve other function classes), coefficient functionals of the form $|a_2a_3-\lambda a_4|$ and $|a_2a_4-\lambda a_3^3|$ (and possibly more) for the defining function classes have frequently appeared to be resolved in the investigations of Hankel determinants for the desired classes of functions.\vskip 2mm

In this paper, we further extend these definitions to include finitely many Fekete-Szeg$\ddot{o}$ parameters $\lambda_j$, $j=1,2,\cdots$ in order to accomodate a wide variety of emerging functionals in the study of coefficients of mappings of the unit disk. Now we say:\vskip 2mm

{\sc Definition 1.} Let $\lambda_i$, $i=1,2,\cdots,q$ be nonnegative real numbers. Then for integers $n\geq 1$ and $q\geq 1$, we define the $q$-th Hankel determinants with Fekete-Szeg$\ddot{o}$ parameters $\lambda_i$, that is $H_q^{\lambda_1,\lambda_2,\cdots,\lambda_q}(n)$, as
$$H_q^{\lambda_1,\lambda_2,\cdots,\lambda_q}(n)=
\begin{vmatrix}
\lambda_1a_n & \lambda_2a_{n+1} & \cdots & \lambda_q a_{n+q-1}\\
a_{n+1} & \cdots & \cdots & \vdots\\
\vdots & \vdots & \vdots & \vdots\\
a_{n+q-1} & \cdots & \cdots & a_{n+2(q-1)}
\end{vmatrix}$$ 
and we also say:\vskip 2mm

{\sc Definition 2.} Let $\lambda_j$, $j=1,2,\cdots,n$ be nonnegative real numbers. Then for integers $n\geq 1$ and $q\geq 1$, we define the $B_q^{\lambda_1,\lambda_2,\cdots,\lambda_n}(n)$ determinants as
$$B_q^{\lambda_1,\lambda_2,\cdots,\lambda_n}(n)=
\begin{vmatrix}
a_n & a_{n+1} & \cdots & \lambda_1a_{n+q-1}\\
a_{n+q} & a_{n+q+1} & \cdots & \lambda_2a_{n+2q-1}\\
a_{n+2q} & a_{n+2q+1} & \cdots & \lambda_3a_{n+3q-1}\\
\vdots & \vdots & \vdots & \vdots\\
a_{n+q(q-1)} & \cdots & \cdots & \lambda_n a_{n+q^2-1}
\end{vmatrix}.$$

For $\lambda_j=1$, $j=1,2,\cdots,q-1 (n-1)$, we simply write $H_q^\lambda(n)$ and $B_q^\lambda(n)$ in place of $H_q^{1,1,\cdots,1,\lambda_q}(n)$ and $B_q^{1,1,\cdots,1,\lambda_n}(n)$ respectively, and we quickly note that for real numbers, $\gamma$, $\alpha$ and $\beta$ we have:\vskip 2mm 

$$H_2^\gamma(1)=\begin{vmatrix}
1 & \gamma a_2\\
a_2 & a_3
\end{vmatrix}=a_3-\gamma a_2^2,$$
$$H_2^\alpha(2)=\begin{vmatrix}
a_2 & \alpha a_3\\
a_3 & a_4
\end{vmatrix}=a_2a_4-\alpha a_3^2$$ and 
$$B_2^\beta(1)=\begin{vmatrix}
1 & a_2\\
a_3 & \beta a_4
\end{vmatrix}=a_2a_3-\beta a_4.$$ 

Then, in this paper we shall investigate the determinant $H_3^{\lambda_1,\lambda_2,\lambda_3}(1)$ for the favoured and well known class of starlike functions, for which Re $zf'(z)/f(z)>0$, denoted by $S^\ast$. By definition 
$$H_3^{\lambda_1,\lambda_2,\lambda_3}(1)=
\begin{vmatrix}
\lambda_1a_1 & \lambda_2a_2 & \lambda_3a_3\\
a_2 & a_3 & a_4\\
a_3 & a_4 & a_5
\end{vmatrix}.$$
For $f\in S^\ast$, $a_1=1$ so that
$$\aligned
H_3^{\lambda_1,\lambda_2,\lambda_3}(1)
&=a_3(\lambda_2a_2a_4-\lambda_3a_3^2)-a_4(\lambda_1a_4-\lambda_3a_2a_3)+a_5(\lambda_1a_3-\lambda_2a_2^2)\\
&=\lambda_2a_3(a_2a_4-\alpha a_3^2)+\lambda_3a_4(a_2a_3-\beta a_4)+\lambda_1a_5(a_3-\gamma a_2^2)
\endaligned$$
where $\alpha=\lambda_3/\lambda_2$, $\beta=\lambda_1/\lambda_3$ and $\gamma=\lambda_2/\lambda_1$. Thus by triangle inequality, we have
$$|H_3^{\lambda_1,\lambda_2,\lambda_3}(1)|\leq\lambda_1|a_5||H_2^\gamma(1)|+\lambda_2|a_3||H_2^\alpha(2)|+\lambda_3|a_4||B_2^\beta(1)|.\eqno{(1.2)}$$
\vskip 2mm

Keen attention is paid to the dynamics of the Fekete-Szeg$\ddot{o}$ parameters in this paper, which leads to ease of analysis and better precision. For our choice class, $S^\ast$, the Fekete-Szeg$\ddot{o}$ functional, $|H_2^\gamma(1)|$, is known and given as:

\begin{theorem}[\cite{FR}]
Let $f\in S^\ast$. Then for real number $\gamma$, $$|H_2^\gamma(1)|\leq\max\{1,|4\gamma-3|\}.$$ The inequality is sharp. For each $\gamma$, equality is attained by $f(z)$ given by
$$f(z)=
\begin{cases}
\frac{z}{1-z^2}&\mbox if\quad \frac{1}{2}\leq\gamma\leq 1,\\
\frac{z}{(1-z)^2}&\mbox if\quad\gamma\in\left[0,\frac{1}{2}\right]\cup[1,\infty).
\end{cases}$$
\end{theorem}

Thus we shall in this paper obtain the best possible bounds on $H_2^\alpha(2)$ and $B_2^\beta(1)$ to conclude our investigation. In the next section we state the lemmas we shall use to establish the desired bounds in Section 3.
 
\section{Preliminary Lemmas}

Let $P$ denote the class of functions $p(z)=1+c_1z+c_2z^2+\cdots$ which are regular in $E$ and satisfy Re $p(z)>0$, $z\in E$. To prove the main results in the next section we shall require the following two lemmas.\vskip 2mm

\begin{lemma} [\cite{PL}]
Let $p\in P$, then $|c_k|\leq 2$, $k=1,2,\cdots$, and the inequality is sharp. Equality is realized by the M$\ddot{o}$bius function $L_0(z)=(1+z)/(1-z)$.
\end{lemma}

\begin{lemma} [\cite{KT}]
Let $p\in P$. Then 
$$\left|c_2-\sigma\frac{c_1^2}{2}\right|\leq
\begin{cases}
2(1-\sigma), & \mbox{$\sigma\leq 0$},\\
2,&\mbox{$0\leq\sigma\leq 2$},\\
2(\sigma-1),&\mbox{$\sigma\geq 2$}
\end{cases}=2\max\{1,|\sigma-1|\}.$$ The inequality is sharp. For each $\sigma$, equality is attained by $p(z)$ given by
$$p(z)=
\begin{cases}
\frac{1+z^2}{1-z^2}&\mbox if\quad 0\leq\sigma\leq 2,\\
\frac{1+z}{1-z}&\mbox if\quad\gamma\in[-\infty,0]\cup[2,\infty).
\end{cases}$$
\end{lemma}

\begin{lemma} [\cite{RJ}]
Let $p\in P$, then 
$$2c_2=c_1^2+x(4-c_1^2)\eqno{(2.1)}$$
and
$$4c_3=c_1^3+2xc_1(4-c_1^2)-x^2c_1(4-c_1^2)+2z(1-|x|^2)(4-c_1^2)\eqno{(2.2)}$$
for some $x$, $z$ such that $|x|\leq 1$ and $|z|\leq 1$.
\end{lemma}

\section{Inequalities for $|H_3^{\lambda_1,\lambda_2,\lambda_3}(1)|$ of $S^\ast$}
First we prove:
\begin{theorem}
Let $f\in S^\ast$. Then for real number $\beta$, $$|B_2^\beta(1)|\leq 2\max\{\beta,|3-2\beta|\}.$$
The inequalities are sharp. For each $\beta$, equality is attained by $f(z)$ given by
$$f(z)=
\begin{cases}
\frac{z}{(1-z^3)^2}&\mbox if\quad 1\leq\beta\leq 3,\\
\frac{z}{(1-z)^2}&\mbox if\quad\beta\in[0,1]\cup[3,\infty).
\end{cases}$$
\end{theorem}

\begin{proof}
It is well known that if $f\in S^\ast$, then $a_2=c_1$, $2a_3=c_2+c_1^2$ and $6a_4=2c_3+3c_1c_2+c_1^3$. Thus we have
$$|B_2^\beta(1)|=|a_2a_3-\beta a_4|=\left|\frac{3-\beta}{6}c_1^3+\frac{1-\beta}{2}c_1c_2-\beta\frac{c_3}{3}\right|.\eqno{(3.1)}$$
First observe that if $\beta\leq 1$, then (3.1) yields
$$|B_2^\beta(1)|\leq\left|\frac{3-\beta}{6}c_1^3-\beta\frac{c_3}{3}\right|+\left|\frac{1-\beta}{2}c_1c_2\right|$$
which, by Lemma 1 gives
$$|B_2^\beta(1)|\leq\left|\frac{3-\beta}{6}c_1^3-\beta\frac{c_3}{3}\right|+2(1-\beta).\eqno{(3.2)}$$
Now substituting for $c_3$ using Lemma 3, we have
$$\aligned \left|\frac{3-\beta}{6}c_1^3-\beta\frac{c_3}{3}\right|
&=\left|\frac{(2-\beta)c_1^3}{4}-\frac{\beta c_1(4-c_1^2)x}{6}+\frac{\beta c_1(4-c_1^2)^2x^2}{12}\right.\\
&\left.-\frac{\beta(4-c_1^2)(1-|x|^2)z}{6}\right|.
\endaligned$$
By Lemma 1 again, $|c_1|\leq 2$. Then letting $c_1=c$, we may assume without restriction that $c\in [0,2]$ so that
$$\aligned \left|\frac{3-\beta}{6}c_1^3-\beta\frac{c_3}{3}\right|\leq
&\frac{(2-\beta)c^3}{4}+\frac{\beta(4-c^2)}{6}+\frac{\beta c(4-c^2)\rho}{6}+\frac{\beta (c-2)(4-c^2)\rho^2}{12}\\
&=F(\rho).
\endaligned$$
Since $\beta\leq 1$, the extreme points of $F(\rho)$ are $\rho=0$, $\rho=1$ and $\rho=c/(2-c)$ (with $c\in[0,1]$ in this case since $\rho\in[0,1]$). Now let $$G_1(c)=F(0)=\frac{(2-\beta)c^3}{4}+\frac{\beta(4-c^2)}{6}$$
\begin{multline*}
G_2(c)=F(1)=\frac{(2-\beta)c^3}{4}+\frac{\beta(4-c^2)}{6}+\frac{\beta c(4-c^2)}{6}+\frac{\beta(c-2)(4-c^2)}{12}\\
=\frac{(1-\beta)c^3}{2}+\beta c
\end{multline*} and
\begin{multline*}
G_3(c)=F\left(\frac{c}{2-c}\right)=\frac{(2-\beta)c^3}{4}+\frac{\beta(4-c^2)}{6}+\frac{\beta c^2(c+2)}{12}\\
=\frac{(3-\beta)c^3}{6}+\frac{2\beta}{3},\;c\in[0,1].
\end{multline*}
By elementary calculus, we find that $G_1(c)\leq G_2(c)\leq G_c(2)=4-2\beta$ while $G_3(c)\leq G_3(1)=(1+\beta)/2$. Hence for $\beta\leq 1$ the maximum of the functional is $4-2\beta$. Using this in (3.3) we have $|B_2^\beta(1)|\leq 6-4\beta$ for $0\leq\beta\leq 1$.\vskip 2mm

Next we consider the case $1\leq\beta\leq 3$. Then we write (3.1) 
$$|B_2^\beta(1)|=\left|\beta\frac{c_3}{3}+\frac{\beta-1}{2}c_1c_2-\frac{\beta-3}{6}c_1^3\right|$$
which gives
$$|B_2^\beta(1)|\leq\beta\frac{|c_3|}{3}+\frac{\beta-1}{2}|c_1|\left|c_2-\frac{\beta-3}{3(\beta-1)}\frac{c_1^2}{2}\right|.$$
Observing that $(\beta-3)/(3\beta-3)\leq 0$ and applying Lemmas 1 and 2, we have $|B_2^\beta(1)|\leq 2\beta$ for $1\leq\beta\leq 3$.\vskip 2mm

Finally, if $\beta\geq 3$ we write (3.1) as
$$|B_2^\beta(1)|=\left|\beta\frac{c_3}{3}+\frac{\beta-1}{2}c_1c_2+\frac{\beta-3}{6}c_1^3\right|$$
and apply Lemma 1 to obtain $|B_2^\beta(1)|\leq 4\beta-6$ in this case. Hence we have 
$$|B_2^\beta(1)|\leq
\begin{cases}
6-4\beta&\mbox if\quad 0\leq\beta\leq 1,\\
2\beta&\mbox if\quad 1\leq\beta\leq 3,\\
4\beta-6&\mbox if\quad\beta\geq 3.
\end{cases}$$
This completes the proof.
\end{proof}

\begin{corollary}[\cite{KO}]
Let $f\in S^\ast$. Then
$$|B_2(1)|\leq 2.$$
The inequality is sharp. Equality is attained by the Koebe function $k(z)=z/(1-z)^2$.
\end{corollary}

Next we have
 
\begin{theorem}
Let $f\in S^\ast$. Then for real number $\alpha$,
$$|H_2^\alpha(2)|\leq\max\{1,|9\alpha-8|\}.$$ The inequalities are sharp. For each $\alpha$, equality is attained by $f(z)$ given by
$$f(z)=
\begin{cases}
\frac{z}{1-z^2/\sqrt{\alpha}}&\mbox if\quad\frac{2}{3}<\alpha\leq 1,\\
\frac{z}{(1-z)^2}&\mbox if\quad\alpha\in\left[0,\frac{2}{3}\right]\cup[1,\infty).
\end{cases}$$
\end{theorem}

\begin{proof}
Following from the proof of Theorem 2, if $f\in S^\ast$, then $a_2=c_1$, $2a_3=c_2+c_1^2$ and $6a_4=2c_3+3c_1c_2+c_1^3$. Thus we have 
$$|H_2^\alpha(2)|=|a_2a_4-\alpha a_3^2|=\left|\frac{c_1c_3}{3}+(1-\alpha)\frac{c_1^2c_2}{2}+(2-3\alpha)\frac{c_1^4}{12}-\alpha\frac{c_2^2}{4}\right|.\eqno{(3.3)}$$
First we suppose $\alpha\leq 2/3$. Then we have
$$|H_2^\alpha(2)|\leq\frac{|c_1||c_3|}{3}+(2-3\alpha)\frac{|c_1|^4}{12}+\frac{\alpha|c_2|}{4}\left|c_2-\frac{4(1-\alpha)}{\alpha}\frac{c_1^2}{2}\right|$$
which, by Lemmas 1 and 2 noting that $4(1-\alpha)/\alpha\geq 2$ for $\alpha\leq 2/3$, gives $|H_2^\alpha(2)|\leq 8-9\alpha$.\vskip 2mm

Substituting for $c_2$ and $c_3$ in (3.3) using Lemma 3 we obtain 
$$\aligned|H_2^\alpha(2)|=
&\left|\frac{(8-9\alpha)c_1^4}{16}+\frac{(10-9\alpha)c_1^2(4-c_1^2)x}{24}-\frac{c_1^2(4-c_1^2)x^2}{12}\right.\\
&\left.+\frac{c_1(1-|x|^2)(4-c_1^2)z}{6}-\frac{\alpha(4-c_1^2)x^2}{16}\right|\endaligned\eqno{(3.4)}$$
Next we write (3.4) as:
$$\aligned|H_2^\alpha(2)|=
&\left|\frac{(9\alpha-8)c_1^4}{16}-\frac{(10-9\alpha)c_1^2(4-c_1^2)x}{24}+\frac{c_1^2(4-c_1^2)x^2}{12}\right.\\
&\left.-\frac{c_1(1-|x|^2)(4-c_1^2)z}{6}+\frac{\alpha(4-c_1^2)x^2}{16}\right|.\endaligned$$
Letting $c_1=c$, assuming without restriction that $c\in [0,2]$ and setting $\rho=|x|$, we have
$$\aligned|H_2^\alpha(2)|\leq
&\frac{(9\alpha-8)c^4}{16}+\frac{c(4-c^2)}{6}+\frac{(10-9\alpha)c^2(4-c^2)\rho}{24}\\
&\frac{(4-c^2)(c-2)[(4-3\alpha)c-6\alpha]\rho^2}{48}=F(\rho).\endaligned$$
Then
$$F'(\rho)=\frac{(10-9\alpha)c^2(4-c^2)}{24}+\frac{(4-c^2)(c-2)[(4-3\alpha)c-6\alpha]\rho}{24}$$
which is positive whenever $2/3\leq\alpha\leq 10/9$. Hence for $2/3\leq\alpha\leq 10/9$, $F(\rho)$ is increasing on $[0,1]$ so that $F(\rho)\leq F(1)$. Thus we have
$$|H_2^\alpha(2)|\leq F(1)=(\alpha-1)c^4-2(\alpha-1)c^2+\alpha=G(c)$$
Now by elementary calculus we see that the maximum of $G(c)$ on $[0,2]$ occurs at $c=1$ if $2/3\leq\alpha\geq 1$ while it is at $c=2$ for $1\leq\alpha\leq 10/9$ and is thus given by $G(1)=1$ for $2/3\leq\alpha\leq 1$ and $G(2)=9\alpha-8$ for $1\leq\alpha\leq 10/9$.\vskip 2mm
Next suppose $\alpha\geq 10/9$. Then we write (3.4) as
$$\aligned|H_2^\alpha(2)|=
&\left|\frac{(9\alpha-8)c_1^4}{16}+\frac{(9\alpha-10)c_1^2(4-c_1^2)x}{24}+\frac{c_1^2(4-c_1^2)x^2}{12}\right.\\
&\left.-\frac{c_1(1-|x|^2)(4-c_1^2)z}{6}+\frac{\alpha(4-c_1^2)x^2}{16}\right|.\endaligned$$
By similar argument as above we have
$$|H_2^\alpha(2)|\leq F(1)=\frac{3\alpha-2}{12}c^4+\frac{3\alpha-4}{3}c^2+\alpha=G(c)$$
which yields $|H_2^\alpha(2)|\leq 9\alpha-8$ similarly. Hence for $\alpha\geq 1$ we have $|H_2^\alpha(2)|\leq 9\alpha-8$.\vskip 2mm
We thus conclude that $$|H_2^\alpha(2)|\leq
\begin{cases}
8-9\alpha&\mbox if\quad 0\leq\alpha\leq\frac{2}{3},\\
1&\mbox if\quad\frac{2}{3}<\alpha\leq 1,\\
9\alpha-8&\mbox if\quad\alpha\geq 1
\end{cases}$$ which completes the proof.
\end{proof}

\begin{corollary}[\cite{AJ}]
Let $f\in S^\ast$. Then
$$|H_2(2)|\leq 1.$$
The inequality is sharp. Equality is attained by the Koebe function $k(z)=z/(1-z)^2$.
\end{corollary}

Now using Theorems 1 to 3 in (1.2) and setting $\alpha=\lambda_3/\lambda_2$, $\beta=\lambda_1/\lambda_3$ and $\gamma=\lambda_2/\lambda_1$, we obtain the grand inequalities for $|H_3^{\lambda_1,\lambda_2,\lambda_3}(1)|$ as follows: 

\begin{theorem}
Let $f\in S^\ast$. Then for real numbers $\lambda_j$, $j=1,2,3$,
\begin{multline*}
|H_3^{\lambda_1,\lambda_2,\lambda_3}(1)|\leq 5\max\{\lambda_1,|4\lambda_2-3\lambda_1|\}+8\max\{\lambda_1,|3\lambda_3-2\lambda_1|\}\\+3\max\{\lambda_2,|9\lambda_3-8\lambda_2|\}.
\end{multline*}
The inequalities are sharp.
\end{theorem}

We conclude with the following interesting and nice bounds on coefficient determinants of starlike functions.\vskip 2mm

\begin{corollary}
Let $f\in S^\ast$. Then
$$|H_3^{1,1,1}(1)|\leq 16,\;|H_3^{1,1,2}(1)|\leq 81,\;|H_3^{1,2,1}(1)|\leq 51.5,\;|H_3^{2,1,1}(1)|\leq 29,$$
$$|H_3^{1,2,2}(1)|\leq 63,\;|H_3^{2,1,2}(1)|\leq 78,\;|H_3^{2,2,1}(1)|\leq 52.5,\;|H_3^{1,3,2}(1)|\leq 105.$$
All inequalities are best possible in the sense that component inequalities in the last theorem are best possible.
\end{corollary}

The first $|H_3(1)|=|H_3^{1,1,1}(1)|\leq 16$ was reported in \cite{KO}.

\vspace{10pt}

\noindent
{\small
ADDRESSES:\\
Current: Department of Physical Sciences, Al-Hikmah University, Ilorin.\\
Permanent: Department of Mathematics, University of Ilorin, Ilorin.\\
kobabalola@gmail.com}

\end{document}